\documentclass[11pt]{amsart}
\usepackage{amsmath}
\usepackage{amssymb}
\usepackage{graphicx}
\usepackage{eurosym}
\usepackage{amsfonts}
\usepackage{xcolor}
\usepackage{mathtools}

\newtheorem{theorem}{Theorem}[section]

\newtheorem{lemma}[theorem]{Lemma}

\theoremstyle{definition}

\newtheorem{remark}[theorem]{Remark}
\DeclareMathOperator{\LLX}{\mathcal{L}(X)}
\DeclareMathOperator{\LL}{\mathcal{L}}

 \numberwithin{equation}{section}
\email{{r.precup@ictp.acad.ro}} \email{andrei.stan@ubbcluj.ro}
\keywords{evolution equation; control; fixed point; semigroup}
\subjclass[2010]{47J35, 34K35, 47H10}

\begin{document}
\title[On equilibrium in control problems]{On equilibrium in control
problems with applications to evolution systems}
\author[R. Precup]{Radu Precup}
\address[R. Precup]{Faculty of Mathematics and Computer Science and
Institute of Advanced Studies in Science and Technology, Babe\c{s}-Bolyai
University, 400084 Cluj-Napoca, Romania \& Tiberiu Popoviciu Institute of
Numerical Analysis, Romanian Academy, P.O. Box 68-1, 400110 Cluj-Napoca,
Romania}
\author[A. Stan]{Andrei Stan}
\address[A. Stan]{A. Stan, Department of Mathematics, Babe\c{s}-Bolyai
University, 400084 Cluj-Napoca, Romania \& Tiberiu Popoviciu Institute of
Numerical Analysis, Romanian Academy, P.O. Box 68-1, 400110 Cluj-Napoca,
Romania}

\begin{abstract}
In this paper we examine a mutual control problem for systems of two abstract
evolution equations subject to a proportionality final condition. Related
observability and semi-observability problems are discussed. The analysis employs a vector fixed-point approach, using matrices rather than constants, and applies the technique of Bielecki equivalent norms.
\end{abstract}

\maketitle

\section{Introduction}

In a recent paper \cite{ps3}, we introduced the concept of \textit{ mutual
control} related to systems whose unknowns exert some control over each
other. More exactly, we have considered four sets $D_{1}$, $D_{2}$ , $C_{1}$%
, $C_{2}$ with $C_{1}\subset D_{1}\times D_{2}$, $C_{2}\subset D_{2}\times
D_{1}$, a linear space $Z,$ two mappings $E_{1}$, $E_{2}:D_{1}\times
D_{2}\rightarrow Z,$ and the problem 
\begin{equation*}
\begin{cases}
E_{1}(x,y)=0_{Z} \\ 
E_{2}(x,y)=0_{Z} \\ 
(x,y)\in C_{1},(y,x)\in C_{2}.%
\end{cases}%
\end{equation*}%
In this context, $y$ controls the state $x$ in the first equation, and $x$
controls the state $y$ in the second. The controllability conditions on $x$
and $y$ are expressed by the appartenence to the sets $C_{1}$ and $C_{2}$,
respectively. We call this problem the mutual control problem.

One way to solve such a problem, is to incorporate the controllability
conditions into the equations and give the problem a fixed point
formulation 
\begin{equation*}
\left( x,y\right) \in \left( N_{1}(x,y),\ N_{2}(x,y)\right) ,
\end{equation*}
where $N_{1},N_{2}$ are set-valued mappings $N_{1}:D_{1}\times
D_{2}\rightarrow D_{1},$ $N_{2}:D_{1}\times D_{2}\rightarrow D_{2}.$ A
solution $\left( x,y\right) $ of this fixed point equation is said to be a 
\textit{solution} of the mutual control problem, while the problem is said
to be \textit{mutually controllable} if such a solution exists.

Also, in \cite{ps3}, we showed that the concept of mutual controllability is
related to the notion of equilibrium in general, and includes in particular
the concept of a Nash equilibrium (\cite{bg,park,p1,p2,ps,ps2,s,s2,s3}).

In this paper, we illustrate the notion of a mutual control by studying a
system of abstract evolution equations 
\begin{equation}
\begin{cases}
x^{\prime }(t)=A\,x(t)+F\left( x(t),y(t)\right) \\ 
y^{\prime }(t)=A\,y(t)+G\left( x(t),y(t)\right)%
\end{cases}
t\in \lbrack 0,T],  \label{pb abstracta}
\end{equation}%
together with the controllability condition 
\begin{equation}
x\left( T\right) -ax(0)=k\left( y\left( T\right) -by\left( 0\right) \right) .
\label{control condition}
\end{equation}%
Here, $a,b,k>0$ are given numbers, $A$ is a linear operator generating a
semigroup of operators, and $F,G:X^{2}\rightarrow X$ are continuous mappings.

We note that the controllability condition \eqref{control condition} under
consideration is non-standard. Unlike the conventional requirement {}{}\cite%
{c}, \cite{z}, where the final values $x\left( T\right) $ and $y\left(
T\right) $ of the two unknowns are specified, it instead demands a
proportional relationship between their deviations from the initial states $%
x\left( 0\right) $ and $y\left( 0\right) ,$ respectively.

Such a control condition is of interest in dynamics of populations when it
expresses the requirement that at a certain moment $T$ the ratio between
proliferations of two populations, for example prey and predators, should be
the desired $k.$ Similarly, for the control of epidemics, it expresses the
requirement that at some time one reach a certain ratio between the infected
population and that susceptible to infection. Analogous interpretations can
be given in the case of some chemical or medical reaction models.

\ Our results target two aspects:

\begin{description}
\item[a)] The problem as an \textit{observability} one, consisting in
determination of the initial states $x\left( 0\right) $ and $y\left(
0\right) $ from the observable final relations;

\item[b)] A \textit{semi-observability }problem, that is, finding a solution
of the mutual problem when only one of the initial states is given,
consequently leading to the determination of the other initial state.
\end{description}

We solve the problem (\ref{pb abstracta})-(\ref{control condition}) using an equivalent formulation as a fixed-point equation. This approach allows for the application of vector techniques based on fixed-point theorems, Bielecki-type norms, and matrices instead of constants. The advantages of each fixed-point method are emphasized, considering specific assumptions about \( F \) and \( G \) to ensure  uniqueness and localization of the solutions.

\section{Preliminaries}

Let $\left( X,|\cdot |_{X}\right) $ be a Banach space, and let $\LL(X)$ be
the set of all bounded linear operators from $X$ to $X$. Endowed with the
norm 
\begin{equation*}
|U|_{\mathcal{L}(X)}=\sup_{x\in X\setminus \{0\}}\frac{|Ux|_{X}}{|x|_{X}},
\end{equation*}%
$\LLX$ is a Banach space.

\subsection{Abstract evolution equations}

Let $T>0$ and $A\colon D(A)\subset X\rightarrow X$ be the generator of a $%
C_{0}$-semigroup $\{S(t)\,:\,t\geq 0\}$.

A function $u\in C\left( [0,T];X\right) $ is said to be a \textit{mild} 
\textit{solution} of the equation 
\begin{equation*}
u^{\prime }(t)=Au(t)+F(t,u(t))\ \ \ \,(t\in \left[ 0,T\right] ),
\end{equation*}%
if it satisfies 
\begin{equation*}
u\left( t\right) =S\left( t\right) u(0)+\int_{0}^{t}S\left( t-s\right)
F\left( s,u\left( s\right) \right) ds,\,\text{ for all $t\in \lbrack 0,T]$.}
\end{equation*}

Throughout this paper, the number $C_{A}$ stands for upper bound for the
norm in $\mathcal{L}(X)$ of the operators $S(t)$, uniform with respect to $%
t\in \left[ 0,2T\right] $, that is 
\begin{equation}
\left\vert S(t)\right\vert _{\mathcal{L}(X)}\leq C_{A},\text{ }
\label{marginire semigrup}
\end{equation}%
for all $t\in \left[ 0,2T\right] .$ If $A$ generates a semigroup of
contractions, then $C_{A}=1.$

For details about semigroups of linear operators and abstract evolution
equations we refer to the books \cite{mr} and \cite{v}.

\subsection{Bielecki-type norms}

For each number $\theta \geq 0$, on the space $C([0,T];X)$,  we define the 
\textit{Bielecki norm} 
\begin{equation*}
|u|_{\theta }:=\max_{t\in \lbrack 0,T]}e^{-\theta t}|u(t)|_{X}.
\end{equation*}%
Note that when $\theta =0,$ the Bieleki norm $\left\vert \cdot \right\vert _{0}$ is
equivalent to the usual Chebyshev (uniform) norm $\left\vert \cdot \right\vert
_{\infty }$. We mention that with a suitable choice of $\theta $, the Bielecki-type norms
allow us to relax the strict conditions on constants such as Lipschitz or
growth conditions required by fixed point theorems.

\subsection{Matrices convergent to zero}

Dealing with systems of equations it is convenient (see, e.g., \cite{pmcz,
pcarte}) to use a vector approach based on matrices instead of constants.

A square matrix $M\in \mathcal{M}_{n\times n}\left( \mathbb{R}_{+}\right) $
is said to be \textit{convergent to zero} if its power $M^{k}$ tends to the
zero matrix as $k\rightarrow \infty $. The next lemma provides equivalent
conditions for a square matrix to be convergent to zero (see, e.g., \cite%
{pmcz}).

\begin{lemma}
\label{caracterizare} Let $M\in \mathcal{M}_{n\times n}\left( \mathbb{R}%
_{+}\right) $ be a square matrix. The following statements are equivalent:
\end{lemma}

\begin{description}
\item[(a)] The matrix $M$ is convergent to zero.

\item[(b)] The spectral radius of $M$ is less than $1$, i.e., $\rho (M)<1$.

\item[(c)] The matrix $I-M,$ where $I$ is the unit matrix of the same size,
is invertible and its inverse has nonnegative entries, i.e., $\left(
I-M\right) ^{-1}\in \mathcal{M}_{n\times n}\left( \mathbb{R}_{+}\right) .$
\end{description}

In case $n=2$, we have the following characterization.

\begin{lemma}
\label{lema2 matrici convergente la zero} A square matrix $M=[a_{ij}]_{1\leq
i,j\leq 2}\in \mathcal{M}_{2\times 2}\left( \mathbb{R}_{+}\right) $ is
convergent to zero if and only if $\ $%
\begin{equation*}
a_{11},\ \ a_{22}<1
\end{equation*}%
and 
\begin{equation*}
\text{tr}(M)<1+\text{det}(M),\,\ \text{i.e.,}\,\
a_{11}+a_{22}<1+a_{11}a_{22}-a_{12}a_{21}.
\end{equation*}
\end{lemma}

In the next Section 3, we use the matrix 
\begin{equation*}
M(\theta )=%
\begin{bmatrix}
a_{11} & a_{12}\frac{e^{\theta T}-1}{\theta } \\[3pt] 
a_{21} & a_{22}\frac{1-e^{-\theta T}}{\theta }%
\end{bmatrix}%
,
\end{equation*}%
where $\theta \geq 0,$ and we aim to find $\theta $ such that $M(\theta )$
is convergent to zero. Here, $a_{ij}\,(i,j=1,2)$ are nonnegative numbers
with $\ a_{11}<1$ and $a_{22}<\frac{1}{T}.$ Notice that the last inequality
guarantees 
\begin{equation*}
a_{22}\frac{1-e^{-\theta T}}{\theta }<1
\end{equation*}%
since $\frac{1-e^{-\theta T}}{\theta }\leq T$ for every $\theta \geq 0.$

The next lemma proved in \cite{ps3} deals with the existence of a $\theta\geq 0 $
for which matrix $M\left( \theta \right) $ is convergent to zero. It shows
us that there can be values {}{}of $\theta >0$ for which the matrix $M\left(
\theta \right) $ is convergent to zero, although the matrix $M(0)$
corresponding to the Chebyshev norm is not convergent to zero.

From Lemma \ref{lema2 matrici convergente la zero}, the matrix $M(\theta )$
is convergent to zero if and only if $h(\theta )<0$, where 
\begin{align*}
h(\theta )& =\text{tr}(M(\theta ))-1-\text{det}(M(\theta )) \\
& =a_{11}+a_{22}\frac{1-e^{-\theta T}}{\theta }-1-a_{11}a_{22}\frac{%
1-e^{-\theta T}}{\theta }+a_{12}a_{21}\frac{e^{\theta T}-1}{\theta }\ \
\left( \theta \geq 0\right) .
\end{align*}

\begin{lemma}
Assume $0 \leq a_{11} < 1$ and $0 < a_{22} < \frac{1}{T}$.

\begin{itemize}
\item[(i)] If $h(0)<0$, then $M(0)$ converges to zero.

\item[(ii)] If $h(0)\geq 0$, then there exists $\theta _{1}>0$ with $%
h^{\prime }(\theta _{1})=0$ and

\begin{itemize}
\item[(a)] if $h(\theta _{1})<0$, then the matrix $M(\theta )$ converges to
zero for every $\theta $ between the zeroes of $h$ and does not converge to
zero otherwise;

\item[(b)] if $h(\theta _{1})\geq 0$, then there are no $\theta $ such that $%
M(\theta )$ converges to zero.
\end{itemize}
\end{itemize}
\end{lemma}

\subsection{Fixed point theorems}

Next, we recall two fixed point theorems, which, along with the well-known Schauder fixed point theorem, will play a key role in our analysis. The first result is Perov's fixed point theorem (see, e.g., \cite{pcarte}) for mappings on the Cartesian product of two metric spaces.

\begin{theorem}[Perov]
Let $\left( X_{i},d_{i}\right) ,$ $i=1,2$ be complete metric spaces and $%
N_{i}:X_{1}\times X_{2}\rightarrow X_{i}$ be two mappings for which there
exists a square matrix $M$ of size two with nonnegative entries and the
spectral radius $\rho \left( M\right) <1$ such that the following vector
inequality 
\begin{equation*}
\left( 
\begin{array}{c}
d_{1}\left( N_{1}\left( x,y\right) ,N_{1}\left( u,v\right) \right) \\ 
d_{2}\left( N_{2}\left( x,y\right) ,N_{2}\left( u,v\right) \right)%
\end{array}%
\right) \leq M\left( 
\begin{array}{c}
d_{1}\left( x,u\right) \\ 
d_{2}\left( y,v\right)%
\end{array}%
\right)
\end{equation*}%
holds for all $\left( x,y\right) ,\left( u,v\right) \in X_{1}\times X_{2}.$
Then, there exists a unique point $\left( x,y\right) \in X_{1}\times X_{2}$
with 
\begin{equation*}
(x,y)=\left( N_{1}\left( x,y\right) ,N_{2}\left( x,y\right) \right) .
\end{equation*}
\end{theorem}

The second result is Avramescu's fixed point theorem \cite{a}.

\begin{theorem}[Avramescu]
Let $D_{1}$ be a closed convex subset of a normed space $Y$, $(D_{2},d)$ a
complete metric space, and let $N_{i}:D_{1}\times D_{2}\rightarrow D_{i}$, $%
i=1,2$ be continuous mappings. Assume that the following conditions are
satisfied:

\begin{enumerate}
\item[(a)] $N_{1}(D_{1}\times D_{2})$ is a relatively compact subset of $Y$;

\item[(b)] There is a constant $L\in \lbrack 0,1)$ such that 
\begin{equation*}
d(N_{2}(x,y),N_{2}(x,y^{\prime }))\leq L\,d(y,y^{\prime })
\end{equation*}%
for all $x\in D_{1}$ and $y,y^{\prime }\in D_{2}$.
\end{enumerate}

Then, there exists $(x,y)\in D_{1}\times D_{2}$ such that 
\begin{equation*}
N_{1}(x,y)=x,\quad N_{2}(x,y)=y.
\end{equation*}
\end{theorem}

Some reference works in fixed point theory are the books \cite{d}, \cite{g}
and \cite{kr}.

\section{Mutual control for abstract evolution equations}

In this section, we aim to find a mild solution of the control problem %
\eqref{pb abstracta}-(\ref{control condition}). More exactly, we look for $%
x,y\in C\left( \left[ 0,T\right] ;X\right) $ such that, for all $t\in
\lbrack 0,T]$, the following relations hold 
\begin{equation}\label{el}
\begin{cases}
x\left( t\right) =S\left( t\right) x\left( 0\right) +\int_{0}^{t}S\left(
t-s\right) F\left( x\left( s\right) ,y\left( s\right) \right) ds \\ 
y\left( t\right) =S\left( t\right) y\left( 0\right) +\int_{0}^{t}S\left(
t-s\right) G\left( x\left( s\right) ,y\left( s\right) \right) ds,%
\end{cases}
\end{equation}%
\begin{equation}
x\left( T\right) -ax\left( 0\right) =k\left( y\left( T\right) -by\left(
0\right) \right) .  \label{3.1'}
\end{equation}

\subsection{Fixed point formulation.}

Our first aim is to combine the controllability condition (\ref{3.1'}) with (%
\ref{el}) into a single system that takes the form of a fixed-point
equation. Assume $\left( x,y\right) $ is a solution of the mutual control
problem. Then, we have 
\begin{align*}
& x\left( T\right) -ax\left( 0\right)  =S\left( T\right) x\left( 0\right)
-ax\left( 0\right) +\int_{0}^{T}S\left( T-s\right) F\left( x\left( s\right)
,y\left( s\right) \right) ds,   \\&
y\left( T\right) -by\left( 0\right) =S\left( T\right) y\left( 0\right)
-by\left( 0\right) +\int_{0}^{T}S\left( T-s\right) G\left( x\left( s\right)
,y\left( s\right) \right) ds,\notag
\end{align*}%
and using (\ref{3.1'}) yields 
\begin{align}
0 &=S\left( T\right) x\left( 0\right) -kS\left( T\right) y\left( 0\right)
-ax\left( 0\right) +kby\left( 0\right)   \label{eq1} \\
&\quad +\int_{0}^{T}S\left( T-s\right) \left( F\left( x\left( s\right) ,y\left(
s\right) \right) -kG\left( x\left( s\right) ,y\left( s\right) \right)
\right) ds.  \notag
\end{align}%
When we express $x\left( 0\right) $ and $y\left( 0\right) $ from $\left( \ref%
{eq1}\right) $ and replace them in (\ref{el}), we obtain%
\begin{equation}  \label{s2}
    \begin{cases}
x\left( t\right)  =N_{1}\left( x,y\right) \left( t\right)  \\
y\left( t\right)  =N_{2}\left( x,y\right) \left( t\right) ,
\end{cases}
\end{equation}%
where $N_{1},N_{2}\colon C\left( \left[ 0,T\right] ;X\right) ^{2}\rightarrow
C\left( \left[ 0,T\right] ;X\right) ^{2}$ are given by%
\begin{align*}
N_1(x,y)(t) &= \frac{1}{a} S(t) \left[ S(T)x(0) - kS(T)y(0) + kb y(0) \right] \\
&\quad + \frac{1}{a} \int_0^T S(T-s) \left( F(x(s), y(s)) - kG(x(s), y(s)) \right) ds \\
&\quad + \int_0^t S(t-s) F(x(s), y(s)) ds, \\
N_2(x,y)(t) &= \frac{1}{kb} S(t) \left[ -S(T)x(0) + kS(T)y(0) + a x(0) \right] \\
&\quad - \frac{1}{kb} \int_0^T S(T-s) \left( F(x(s), y(s)) - kG(x(s), y(s)) \right) ds \\
&\quad + \int_0^t S(t-s) G(x(s), y(s)) ds.
\end{align*}
Thus, the pair $\left( x,y\right) $ is a fixed point of the operator $\left(
N_{1},N_{2}\right) $. In the following lemma, we show that the converse also
holds, i.e., any fixed point of the operator $\left( N_{1},N_{2}\right) $
satisfies both $\left( \ref{el}\right) $ and $\left( \ref{3.1'}\right) $.

\begin{lemma}
Any fixed point of the operator $(N_1,N_2)$ satisfy  $\left( \ref{el}\right) $ and $\left( \ref{3.1'}\right) $.
\end{lemma}

\begin{proof}
Assume that $\left( x,y\right) $ is a solution of (\ref{s2}). Setting $t=0$,
we obtain%
\begin{eqnarray*}
ax\left( 0\right)  &=&S\left( T\right) x\left( 0\right) -kS\left( T\right)
y\left( 0\right) +kby\left( 0\right)  \\
&&+\int_{0}^{T}S\left( T-s\right) \left( F\left( x\left( s\right) ,y\left(
s\right) \right) -kG\left( x\left( s\right) ,y\left( s\right) \right)
\right) ds, \\
kby\left( 0\right)  &=&-S\left( T\right) x\left( 0\right) +kS\left( T\right)
y\left( 0\right) +ax\left( 0\right)  \\
&&-\int_{0}^{T}S\left( T-s\right) \left( F\left( x\left( s\right) ,y\left(
s\right) \right) -kG\left( x\left( s\right) ,y\left( s\right) \right)
\right) ds,
\end{eqnarray*}%
whence%
\begin{align*}
ax\left( 0\right) -kby\left( 0\right)  &=2S\left( T\right) x\left( 0\right)
-2kS\left( T\right) y\left( 0\right) -\left( ax\left( 0\right) -kby\left(
0\right) \right)  \\&
\quad+2\int_{0}^{T}S\left( T-s\right) \left( F\left( x\left( s\right) ,y\left(
s\right) \right) -kG\left( x\left( s\right) ,y\left( s\right) \right)
\right) ds,
\end{align*}%
that is%
\begin{align} \label{star}
 ax\left( 0\right) -kby\left( 0\right)  &=S\left( T\right) x\left( 0\right)
-kS\left( T\right) y\left( 0\right)   \\&
\quad +\int_{0}^{T}S\left( T-s\right) \left( F\left( x\left( s\right) ,y\left(
s\right) \right) -kG\left( x\left( s\right) ,y\left( s\right) \right)
\right) ds.  \notag
\end{align}%
Letting $t=T$ in (\ref{s2}) gives%
\begin{align*}
x\left( T\right) &= \frac{1}{a} \left[ S\left( T\right) \left( S\left( T\right) x\left( 0\right) - k S\left( T\right) y\left( 0\right) + k b y\left( 0\right) \right. \right] \\
&\quad   +  \frac{1}{a} S\left( T\right) \int_{0}^{T} S\left( T - s\right) \left( F\left( x\left( s\right), y\left( s\right) \right) - k G\left( x\left( s\right), y\left( s\right) \right) \right) ds  \\
&\quad + \int_{0}^{T} S\left( T - s\right) F\left( x\left( s\right), y\left( s\right) \right) ds,
\end{align*}
and
\begin{align*}
k y\left( T\right) &= \frac{1}{b} \left[ S\left( T\right) \left( -S\left( T\right) x\left( 0\right) + k S\left( T\right) y\left( 0\right) + a x\left( 0\right) \right. \right] \\
&\quad -  \frac{1}{b} S\left( T\right)\int_{0}^{T} S\left( T - s\right) \left( F\left( x\left( s\right), y\left( s\right) \right) - k G\left( x\left( s\right), y\left( s\right) \right) \right) ds  \\
&\quad + k \int_{0}^{T} S\left( T - s\right) G\left( x\left( s\right), y\left( s\right) \right) ds.
\end{align*}
Whence%
\begin{align*}
x\left( T\right) - ky\left( T\right) &= \frac{1}{a} S\left( T\right) \big [ S\left( T\right) x\left( 0\right) - k S\left( T\right) y\left( 0\right) + k k b y\left( 0\right) \\
&\quad + \int_{0}^{T} S\left( T - s\right) \left( F\left( x\left( s\right), y\left( s\right) \right) - k G\left( x\left( s\right), y\left( s\right) \right) \right) ds \big ] \\
&\quad + \int_{0}^{T} S\left( T - s\right) F\left( x\left( s\right), y\left( s\right) \right) ds \\
&\quad + \frac{1}{b} S\left( T\right) \big[ S\left( T\right) x\left( 0\right) - k S\left( T\right) y\left( 0\right) - a x\left( 0\right)  \\
&\quad  + \int_{0}^{T} S\left( T - s\right) \left( F\left( x\left( s\right), y\left( s\right) \right) - k G\left( x\left( s\right), y\left( s\right) \right) \right) ds \big] \\
&\quad - k \int_{0}^{T} S\left( T - s\right) G\left( x\left( s\right), y\left( s\right) \right) ds.
\end{align*}
Since, by (\ref{eq1}), the expressions within the square brackets equal \(ax\left(0\right)\) and \(-kby\left(0\right)\), respectively, we deduce:
\begin{align}
x\left( T\right) -ky\left( T\right)  &=S\left( T\right) x\left( 0\right)
-kS\left( T\right) y\left( 0\right)   \label{star2} \\
&+\int_{0}^{T}S\left( T-s\right) \left( F\left( x\left( s\right) ,y\left(
s\right) \right) -kG\left( x\left( s\right) ,y\left( s\right) \right)
\right) ds.  \notag
\end{align}%
Now (\ref{star}) and (\ref{star2}) imply 
\begin{equation*}
ax\left( 0\right) -kby\left( 0\right) =x\left( T\right) -ky\left( T\right),
\end{equation*}%
and so the controllability condition is satisfied.

Finally, (\ref{star}) and (\ref{s2}) imply that $\left( x,y\right) $
satisfies (\ref{el}).
\end{proof}


\subsection{Existence}

Using Perov's and Leray-Schauder's fixed point theorems, we obtain the
following existence result.

\begin{theorem}
\label{teorema cu Perov} Assume that the following conditions are satisfied:

\begin{itemize}
\item[(i) ] There are constants $a_{F},b_{F},a_{G},b_{G}\geq 0$ such that 
\begin{align}
& |F(x,y)-F(\overline{x},\overline{y})|_{X}\leq a_{F}|x-\overline{x}%
|_{X}+b_{F}|y-\overline{y}|_{X}, \\
& |G(x,y)-G(\overline{x},\overline{y})|_{X}\leq a_{G}|x-\overline{x}%
|_{X}+b_{G}|y-\overline{y}|_{X},  \notag
\end{align}%
for all $x,\overline{x},y,\overline{y}\in X$ and $t\in \left[ 0,T\right] ;$

\item[(ii)] The map $H:=F-kG$ is bounded on $X^{2};$

\item[(iii)] The matrix 
\begin{equation}
M:=TC_{A}\left[ 
\begin{array}{ll}
\frac{1}{a}C_{A}\left( a_{F}+ka_{G}\right) +a_{F} & \frac{1}{a}C_{A}\left(
b_{F}+kb_{G}\right) +b_{F}\smallskip \\ 
\frac{1}{kb}C_{A}\left( a_{F}+ka_{G}\right) +a_{G} & \frac{1}{kb}C_{A}\left(
b_{F}+kb_{G}\right) +b_{G}%
\end{array}%
\right]  \label{r0}
\end{equation}%
is convergent to zero;

\item[(iv)] Semigroup $\left\{ S\left( t\right) ;\ t\geq 0\right\} $ is
compact and $\left\vert S\left( T\right) \right\vert _{\mathcal{L}\left(
X\right) }<\frac{2}{\frac{1}{a}+\frac{1}{b}}.$
\end{itemize}

Then problem \emph{(\ref{el})-(\ref{3.1'})} has at least one solution $%
\left( x,y\right) $ in $C\left( \left[ 0,T\right] ;X\right) ^{2}.$
\end{theorem}

\begin{proof}
We prove the result in two steps: (a) First, using Perov's fixed point
theorem we prove that for any $\alpha ,\beta \in X,$ there exists in $%
C\left( \left[ 0,T\right] ;X\right) ^{2}$ a unique fixed point $\left(
x_{\alpha ,\beta },y_{\alpha ,\beta }\right) $ of the operator $%
(N_{1},N_{2}) $ defined as follows:%
\begin{align}
N_{1}\left( x,y\right) \left( t\right) &= \frac{1}{a} S\left( t\right)
\left[ S\left( T\right) \alpha - k S\left( T\right) \beta + k b \beta \right] \\
&\quad  + \frac{1}{a}S(t)\int_{0}^{T} S\left( T - s\right) \left( F\left( x\left( s\right), y\left( s\right) \right) - k G\left( x\left( s\right), y\left( s\right) \right) \right) ds  \notag \\
&\quad + \int_{0}^{t} S\left( t - s\right) F\left( x\left( s\right), y\left( s\right) \right) ds. \notag
\end{align}
\begin{align}
N_{2}\left( x,y\right) \left( t\right) &= \frac{1}{k b} S\left( t\right)
\left[ - S\left( T\right) \alpha + k S\left( T\right) \beta + a \alpha \right] \notag \\
&\quad  -  \frac{1}{k b} S\left( t\right)\int_{0}^{T} S\left( T - s\right) \left( F\left( x\left( s\right), y\left( s\right) \right) - k G\left( x\left( s\right), y\left( s\right) \right) \right) ds \notag \\
&\quad + \int_{0}^{t} S\left( t - s\right) G\left( x\left( s\right), y\left( s\right) \right) ds. \notag
\end{align}

For any $x,\overline{x},y,\overline{y}\in C\left( \left[ 0,T\right]
;X\right) ,$ one has 
\begin{align*}
\left\vert N_{1}\left( x,y\right) -N_{1}\left( \overline{x},\overline{y}%
\right) \right\vert _{0} 
&\leq TC_{A}\left( \frac{1}{a}C_{A}\left( a_{F}+ka_{G}\right) +a_{F}\right)
\left\vert x-\overline{x}\right\vert _{0}\\&\quad +TC_{A}\left( \frac{1}{a}%
C_{A}\left( b_{F}+kb_{G}\right) +b_{F}\right) \left\vert y-\overline{y}%
\right\vert _{0}, 
\end{align*}
and
\begin{align*}
\left\vert N_{2}\left( x,y\right) -N_{2}\left( \overline{x},\overline{y}%
\right) \right\vert _{0} 
&\leq TC_{A}\left( \frac{1}{kb}C_{A}\left( a_{F}+ka_{G}\right)
+a_{G}\right) \left\vert x-\overline{x}\right\vert _{0}\\&\quad+TC_{A}\left( \frac{1%
}{kb}C_{A}\left( b_{F}+kb_{G}\right) +b_{G}\right) \left\vert y-\overline{y}%
\right\vert _{0}.
\end{align*}%
Writing the above relations in the vector form, we obtain 
\begin{equation*}
\left[ 
\begin{array}{c}
\left\vert N_{1}\left( x,y\right) -N_{1}\left( \overline{x},\overline{y}%
\right) \right\vert _{0} \\[5pt] 
\left\vert N_{2}\left( x,y\right) -N_{2}\left( \overline{x},\overline{y}%
\right) \right\vert _{0}%
\end{array}%
\right] \leq M\left[ 
\begin{array}{c}
\left\vert x-\overline{x}\right\vert _{0} \\[5pt] 
\left\vert y-\overline{y}\right\vert _{0}%
\end{array}%
\right] .
\end{equation*}%
Consequently, $N=(N_{1},N_{2})$ is a Perov contraction on $C\left( \left[ 0,T%
\right] ;X\right) ^{2}$ and thus it has a unique fixed point.

(b) Using Leray-Schauder's fixed point theorem we prove that the map $%
H:X^{2}\rightarrow X^{2},$%
\begin{equation*}
\left( \alpha ,\beta \right) \mapsto \left( x_{\alpha ,\beta }\left(
0\right) ,y_{\alpha ,\beta }\left( 0\right) \right) ,
\end{equation*}%
has at least one fixed point $\left( \alpha ^{\ast },\beta ^{\ast }\right) .$
Then the pair $\left( x_{\alpha ^{\ast },\beta ^{\ast }},y_{\alpha ^{\ast
},\beta ^{\ast }}\right) $ is a solution of the mutual control problem. To start,  we show that $H$ is continuous. Indeed, denoting%
\begin{align*}
\gamma_1(\alpha, \beta) &= \frac{1}{a}S(t) \left( S(T) \alpha - k S(T) \beta + k b \beta \right), \\
\gamma_2(\alpha, \beta) &= \frac{1}{kb}S(t) \left( -S(T) \alpha + k S(T) \beta + a \alpha \right),
\end{align*}
one has%
\begin{align*}
\left\vert \gamma_1(\alpha, \beta) - \gamma_1(\overline{\alpha}, \overline{\beta}) \right\vert_X &\leq C \left( \left\vert \alpha - \overline{\alpha} \right\vert_X + \left\vert \beta - \overline{\beta} \right\vert_X \right), \\
\left\vert \gamma_2(\alpha, \beta) - \gamma_2(\overline{\alpha}, \overline{\beta}) \right\vert_X &\leq C \left( \left\vert \alpha - \overline{\alpha} \right\vert_X + \left\vert \beta - \overline{\beta} \right\vert_X \right).
\end{align*}
for a suitable constant $C>0.$ Then%
\begin{align*}
\left\vert x_{\alpha, \beta} - x_{\overline{\alpha}, \overline{\beta}} \right\vert_0 
&\leq C \left( \left\vert \alpha - \overline{\alpha} \right\vert_X + \left\vert \beta - \overline{\beta} \right\vert_X \right) \\
&\quad + m_{11} \left\vert x_{\alpha, \beta} - x_{\overline{\alpha}, \overline{\beta}} \right\vert_0 
+ m_{12} \left\vert y_{\alpha, \beta} - y_{\overline{\alpha}, \overline{\beta}} \right\vert_0, \\
\left\vert y_{\alpha, \beta} - y_{\overline{\alpha}, \overline{\beta}} \right\vert_0 
&\leq C \left( \left\vert \alpha - \overline{\alpha} \right\vert_X + \left\vert \beta - \overline{\beta} \right\vert_X \right) \\
&\quad + m_{21} \left\vert x_{\alpha, \beta} - x_{\overline{\alpha}, \overline{\beta}} \right\vert_0 
+ m_{22} \left\vert y_{\alpha, \beta} - y_{\overline{\alpha}, \overline{\beta}} \right\vert_0.
\end{align*}
where $m_{ij}$ are the entries of matrix $M.$ Hence%
\begin{equation*}
\begin{bmatrix}
\left\vert x_{\alpha, \beta} - x_{\overline{\alpha}, \overline{\beta}} \right\vert_0 \\[7pt]
\left\vert y_{\alpha, \beta} - y_{\overline{\alpha}, \overline{\beta}} \right\vert_0
\end{bmatrix} \leq M 
\begin{bmatrix}
\left\vert x_{\alpha, \beta} - x_{\overline{\alpha}, \overline{\beta}} \right\vert_0 \\[7pt]
\left\vert y_{\alpha, \beta} - y_{\overline{\alpha}, \overline{\beta}} \right\vert_0
\end{bmatrix}  +
\begin{bmatrix}
C \left( \left\vert \alpha - \overline{\alpha} \right\vert_X + \left\vert \beta - \overline{\beta} \right\vert_X \right) \\[7pt]
C \left( \left\vert \alpha - \overline{\alpha} \right\vert_X + \left\vert \beta - \overline{\beta} \right\vert_X \right)
\end{bmatrix},
\end{equation*}
equivalently%
\begin{equation*}
\begin{bmatrix}
\left\vert x_{\alpha, \beta} - x_{\overline{\alpha}, \overline{\beta}} \right\vert_0 \\[7pt]
\left\vert y_{\alpha, \beta} - y_{\overline{\alpha}, \overline{\beta}} \right\vert_0
\end{bmatrix} \leq \left( I-M\right) ^{-1}\begin{bmatrix}
C \left( \left\vert \alpha - \overline{\alpha} \right\vert_X + \left\vert \beta - \overline{\beta} \right\vert_X \right) \\[7pt]
C \left( \left\vert \alpha - \overline{\alpha} \right\vert_X + \left\vert \beta - \overline{\beta} \right\vert_X \right)
\end{bmatrix},
\end{equation*}%
which clearly shows that $x_{\alpha ,\beta }$ and $y_{\alpha ,\beta }$
depend continuously on $\alpha $ and $\beta .$ This implies the continuity
of $H.$ Next we show that $H$ maps bounded sets into bounded sets. This
follows from the estimate%
\begin{equation*}
\left[ 
\begin{array}{c}
\left\vert x_{\alpha ,\beta }-x_{0,0}\right\vert _{0}\smallskip \\ 
\left\vert y_{\alpha ,\beta }-y_{0,0}\right\vert _{0}%
\end{array}%
\right] \leq \left( I-M\right) ^{-1}\left[ 
\begin{array}{c}
C\left( \left\vert \alpha \right\vert _{X}+\left\vert \beta \right\vert
_{X}\right) \smallskip \\ 
C\left( \left\vert \alpha \right\vert _{X}+\left\vert \beta \right\vert
_{X}\right)%
\end{array}%
\right] .
\end{equation*}%
Third, since the semigroup is compact, $H$ maps any bounded set into a
compact set. Therefore, the map $H$ is completely continuous. The final step is to establish the boundedness of the set of all solutions to the family of equations
\begin{equation*}
H\left( \alpha ,\beta \right) =\lambda \left( \alpha ,\beta \right) ,\ \ \
\lambda >1.
\end{equation*}%
If $\left( \alpha ,\beta \right) $ is such a solution, then%
\begin{align}
\lambda \alpha &= \frac{1}{a} \big[ S\left( T \right) \alpha - kS\left( T \right) \beta + kb \beta  \notag \\
&\quad + \int_{0}^{T} S\left( T - s \right) \left( F\left( x_{\alpha, \beta}\left( s \right), y_{\alpha, \beta}\left( s \right) \right) - kG\left( x_{\alpha, \beta}\left( s \right), y_{\alpha, \beta}\left( s \right) \right) \right) ds \big], \label{f1}
\end{align}
\begin{align*}
\lambda \beta &= \frac{1}{kb} \big[ -S\left( T \right) \alpha + kS\left( T \right) \beta + a \alpha  \\
&\quad - \int_{0}^{T} S\left( T - s \right) \left( F\left( x_{\alpha, \beta}\left( s \right), y_{\alpha, \beta}\left( s \right) \right) - kG\left( x_{\alpha, \beta}\left( s \right), y_{\alpha, \beta}\left( s \right) \right) \right) ds \big].
\end{align*}
When we add the above two relations, we find
\begin{equation*}
\lambda \left( a\alpha +kb\beta \right) =a\alpha +kb\beta ,
\end{equation*}%
and since \(\lambda > 1\), this implies \(a\alpha + kb\beta = 0\). Substituting $kb\beta $
by $-a\alpha $ in (\ref{f1}) and passing to the norm, in virtue of (ii), we
obtain 
\begin{equation*}
2\left\vert \alpha \right\vert _{X}<\left( \lambda +1\right) \left\vert
\alpha \right\vert _{X}\leq \left( \frac{1}{a}+\frac{1}{b}\right) \left\vert
S\left( T\right) \right\vert _{\mathcal{L}\left( X\right) }\left\vert \alpha
\right\vert _{X}+\widetilde{C},
\end{equation*}%
for some constant $\widetilde{C}.$ Using the condition on $\left\vert
S\left( T\right) \right\vert _{\mathcal{L}\left( X\right) }$ from (iv), we
deduce the boundedness of $\alpha $ independently of $\lambda .$ Next from $%
\alpha \alpha +kb\beta =0,$ we also have that $\beta $ is bounded. Therefore
Leray-Schauder's theorem applies and gives the result.
\end{proof}

\begin{remark}
The previous existence result can be seen as an observability result.
Indeed, having observed the final relation (\ref{3.1'}), we may guess one of
the possible initial states $\left( x\left( 0\right) ,y\left( 0\right)
\right) $ of the process.
\end{remark}

In the next section we deal with a semi-observability problem for the
situation that one of the initial states, say $y\left( 0\right) ,$ is known
while the second one $x\left( 0\right) $ is found after obtaining a solution
of the mutual control problem.

\section{A semi-observability problem}

In this section, we aim to find a mild solution of the control problem %
\eqref{pb abstracta}-(\ref{control condition}) assuming that the initial
state $y\left( 0\right) =\beta $ is known. More exactly, we look for $x,y\in
C\left( \left[ 0,T\right] ;X\right) $ such that, for all $t\in \lbrack 0,T]$%
, the following relations hold: 
\begin{equation}
\begin{cases}
x\left( t\right) =S\left( t\right) x\left( 0\right) +\int_{0}^{t}S\left(
t-s\right) F\left( x\left( s\right) ,y\left( s\right) \right) ds \\ 
y\left( t\right) =S\left( t\right) y\left( 0\right) +\int_{0}^{t}S\left(
t-s\right) G\left( x\left( s\right) ,y\left( s\right) \right) ds,%
\end{cases}
\label{E1}
\end{equation}%
\begin{equation}
y\left( 0\right) =\beta ,  \label{E2}
\end{equation}%
\begin{equation}
x\left( T\right) -ax\left( 0\right) =k\left( y\left( T\right) -by\left(
0\right) \right) .  \label{E3}
\end{equation}

\subsection{Fixed point formulation of the problem}

Similar to the previous section, we try to incorporate all the equations in only one fixed
point problem.
Let $\left( x,y\right) $ be a solution of (\ref{E1})-(\ref{E3}). Then, using
(\ref{E1}) and (\ref{E3}) gives%
\begin{equation*}
S\left( T\right) \left( x\left( 0\right) -k\beta \right)
+\int_{0}^{T}S\left( T-s\right) \left( F\left( x(s),y(s)\right) -kG\left(
x(s),y(s)\right) \right) ds=ax\left( 0\right) -kb\beta ,
\end{equation*}%
whence%
\begin{align}\label{EC}
x\left( 0\right) &= \frac{1}{a} \left( S\left( T\right) \left( x\left( 0\right) - k\beta \right) + kb\beta \right) \\
& \quad + \frac{1}{a} \int_{0}^{T} S\left( T - s \right) \left( F\left( x(s), y(s) \right)ds - kG\left( x(s), y(s) \right) \right) \, ds. \notag
\end{align}
When we substitute it in (\ref{E1}), one has
\begin{equation}
    \begin{cases}
x\left( t\right)  =N_{1}\left( x,y\right) \left( t\right)  \\
y\left( t\right)  =N_{2}\left( x,y\right) \left( t\right) ,
\end{cases}
\end{equation}\label{E0}
where%
\begin{align*}
N_{1}(x,y)(t) &= \frac{1}{a}S(t+T)\left( x(0) - k\beta \right) + k\frac{b}{a}S(t)\beta \\
              &\quad + \int_{0}^{t} S(t-s) F(x(s), y(s))\, ds \\
              &\quad + \frac{1}{a} \int_{0}^{T} S(t+T-s)\left( F(x, y - kG(x, y)) \right), \\
N_{2}(x,y)(t) &= S(t)\beta + \int_{0}^{t} S(t-s) G(x(s), y(s))\, ds.
\end{align*}
The converse implication also holds.

\begin{lemma}
A pair $\left( x,y\right) $ solves the system $(\ref{E0})$ if and only if
it solves $(\ref{E1})$-$(\ref{E3})$.
\end{lemma}
\begin{proof}
We only have to prove that (\ref{E0}) implies (\ref{E1})-(\ref{E3}).
 Clearly, (\ref{E2}) follows directly from the second equation in (\ref{E0}). Setting $%
t=0$ in the first equation of (\ref{E0}) gives (\ref{EC}), while setting $t=T$
provides%
\begin{align*}
x(T) &= S(T) \left[ \frac{1}{a}S(T)\left( x(0) - k\beta \right) + k\frac{b}{a}\beta\right] \\& \quad + \frac{1}{a}S(T)\int_{0}^{T} S(T-s) \left( F(x(s),y(s)) - kG(x(s),y(s)) \right) ds \\
&\quad + \int_{0}^{T} S(T-s) F(x(s),y(s)) ds.
\end{align*}
From (\ref{EC}), one has%
\begin{equation*}
S\left( T\right) x\left( 0\right) =x\left( T\right) -\int_{0}^{T}S\left(
T-s\right) F\left( x(s),y(s)\right) ds,
\end{equation*}
and using it in (\ref{EC}) yields%
\begin{align*}
x(0) &= \frac{1}{a}x(T) - \frac{1}{a}\int_{0}^{T} S(T-s) F(x(s),y(s)) ds - k\frac{1}{a}S(T)\beta + k\frac{b}{a}\beta \\
&\quad + \frac{1}{a}\int_{0}^{T} S(T-s) \left( F(x(s),y(s)) - kG(x(s),y(s)) \right)  ds \\
&= \frac{1}{a}x(T) - k\frac{1}{a}S(T)\beta + k\frac{b}{a}\beta - k\frac{1}{a}\int_{0}^{T} S(T-s) G(x(s),y(s)) ds \\
&= \frac{1}{a}x(T) - k\frac{1}{a}\left[ S(T)\beta + \int_{0}^{T} S(T-s) G(x(s),y(s)) ds \right] + k\frac{b}{a}\beta \\
&= \frac{1}{a}x(T) - k\frac{1}{a}y(T) + k\frac{b}{a}\beta,
\end{align*}
that is (\ref{E3}). Finally, combining (\ref{EC}) and the first equation of (\ref{E0}) results in the first equation of (\ref{E1}).
\end{proof}

\subsection{Existence and uniqueness via Perov's fixed point theorem}

Using Perov's fixed point theorem, we obtain the following existence and
uniqueness result.

\begin{theorem}
\label{t31} Assume that the following conditions are satisfied:

\begin{itemize}
\item[(i) ] There are nonnegative constants $a_{11},a_{12},a_{21},a_{22}$
such that 
\begin{align*}
& |F(x,y)-F(\overline{x},\overline{y})|_{X}\leq a_{11}|x-\overline{x}%
|_{X}+a_{12}|y-\overline{y}|_{X}, \\
& |G(x,y)-G(\overline{x},\overline{y})|_{X}\leq a_{21}|x-\overline{x}%
|_{X}+a_{22}|y-\overline{y}|_{X},
\end{align*}%
for all $x,\overline{x},y,\overline{y}\in X$ and $t\in \left[ 0,T\right] ;$

\item[(ii)] There exists $\theta \geq 0$ such that the matrix 
\begin{equation*}
M(\theta ):=C_{A}%
\begin{bmatrix}
\frac{1}{a}+\left( 1+\frac{1}{a}\right) Ta_{11}+\frac{k}{a}Ta_{21} & \left(
\left( 1+\frac{1}{a}\right) a_{12}+\frac{k}{a}a_{22}\right) \frac{e^{\theta
T}-1}{\theta } \\[3pt] 
Ta_{21} & a_{22}\,\frac{1-e^{-\theta T}}{\theta }%
\end{bmatrix}%
\end{equation*}%
is convergent to zero.
\end{itemize}

Then, for every $\beta \in X$, problem \emph{(\ref{el})-(\ref{3.1'})} has a
unique solution $\left( x,y\right) $ in $C\left( \left[ 0,T\right] ;X\right)
^{2}$ with $y\left( 0\right) =\beta .$
\end{theorem}

\begin{proof}
We apply Perov's fixed point theorem to the operator $(N_{1},N_{2})$ with
the following choice of spaces: $X_{1}=$ $C([0,T];X)$ endowed with the
uniform norm, and $X_{2}=\left\{ y\in
C([0,T];X):\ y\left( 0\right) =\beta \right\} $ endowed with the metric
induced by the Bielecki norm $\left\vert \cdot\right\vert _{\theta }$ on $%
C([0,T];X)$. Clearly, $N_{1}\left( X_{1}\times X_{2}\right) \subset X_{1},$
while the inclusion $N_{2}\left( X_{1}\times X_{2}\right) \subset X_{2}$ is
obvious from the definition of $N_{2}.$ Next, let $x,\overline{x}\in X_{1}$
and $y,\overline{y}\in X_{2}$. One has \ 
\begin{align*}
& \left\vert N_{1}(x,y)\left( t\right) -N_{1}(\overline{x},\overline{y}%
)\left( t\right) \right\vert _{X} \\
& \leq \frac{1}{a}C_{A}\left\vert x-\overline{x}\right\vert _{0}+\left(
\left( 1+\frac{1}{a}\right) C_{A}Ta_{11}+\frac{k}{a}C_{A}Ta_{21}\right)
\left\vert x-\overline{x}\right\vert _{0} \\
& \quad +\left( \left( 1+\frac{1}{a}\right) C_{A}a_{12}+\frac{k}{a}%
C_{A}a_{22}\right) \int_{0}^{T}\left\vert y\left( s\right) -\overline{y}%
\left( s\right) \right\vert _{X}ds \\
& \leq m_{11}\left\vert x-\overline{x}\right\vert _{0}+m_{12}\frac{e^{\theta
T}-1}{\theta }|y-\overline{y}|_{\theta },
\end{align*}%
where%
\begin{align*}
m_{11} &= C_{A}\left( \frac{1}{a} + \left( 1 + \frac{1}{a} \right) Ta_{11} + \frac{k}{a} Ta_{21} \right), \\
m_{12} &= C_{A}\left( \left( 1 + \frac{1}{a} \right) a_{12} + \frac{k}{a} a_{22} \right).
\end{align*}
Therefore, taking the supremum over $t\in \lbrack 0,T]$ yields 
\begin{equation}
\left\vert N_{1}(x,y)-N_{1}(\overline{x},\overline{y})\right\vert _{0}\leq
m_{11}|x-\overline{x}|_{0}+m_{12}\frac{e^{\theta T}-1}{\theta }|y-\overline{y%
}|_{\theta }.  \label{r2}
\end{equation}%
For the second operator $N_{2}$, we estimate 
\begin{align*}
\left\vert N_{2}(x,y)\left( t\right) -N_{2}(\overline{x},\overline{y})\left(
t\right) \right\vert _{X}& \leq \,\,C_{A}Ta_{21}|x-\overline{x}|_{0} \\
& +a_{22}\,C_{A}\int_{0}^{t}|y(s)-\overline{y}(s)|_{X}ds \\
& \leq m_{21}|x-\overline{x}|_{0}+m_{22}\frac{e^{\theta t}-1}{\theta }|y-%
\overline{y}|_{\theta },
\end{align*}%
where 
\begin{equation*}
m_{21}=\,C_{A}Ta_{21},\,\ \ m_{22}=C_{A}a_{22}.
\end{equation*}%
Multiplying the both sides by $e^{-\theta t}$ and taking the supremum over $%
t\in \lbrack 0,T]$, we obtain 
\begin{equation}
\left\vert N_{2}(x,y)-N_{2}(\overline{x},\overline{y})\right\vert _{\theta
}\leq m_{21}|x-\overline{x}|_{0}+m_{22}\frac{1-e^{-\theta T}}{\theta }|y-%
\overline{y}|_{\theta }.  \label{r3}
\end{equation}%
Now, writing inequalities (\ref{r2}) and (\ref{r3}) in a vector form, gives 
\begin{equation*}
\begin{bmatrix}
\left\vert N_{1}(x,y)-N_{1}(\overline{x},\overline{y})\right\vert
_{0}\smallskip \\ 
\left\vert N_{2}(x,y)-N_{2}(\overline{x},\overline{y})\right\vert _{\theta }%
\end{bmatrix}%
\leq M(\theta )%
\begin{bmatrix}
|x-\overline{x}|_{0}\smallskip \\ 
|y-\overline{y}|_{\theta }%
\end{bmatrix}%
,
\end{equation*}%
where from assumption (ii), the matrix $M(\theta )$ is convergent to zero.
Therefore, Perov's fixed point theorem applies and guarantees the existence
of a unique fixed point $\left( x,y\right) \in X_{1}\times X_{2}$ of the
operator $(N_{1},N_{2})$.
\end{proof}

\subsection{Existence via Schauder's fixed point theorem}

The next result does not require Lipschitz conditions for $F$ and $G,$ but
only linear growth ones. Although this sacrifices uniqueness, it offers the advantage of localizing the solution.

\begin{theorem}
\label{teorema cu Schauder} Assume the there are nonnegative constants $%
a_{ij},$ $i=1,2;$ $j=1,2,3$ such that 
\begin{align}
& |F(x,y)|_{X}\leq a_{11}|x|_{X}+a_{12}|y|_{X}+a_{13}, \\
& |G(x,y)|_{X}\leq a_{21}|x|_{X}+a_{22}|y|_{X}+a_{23},  \notag
\end{align}%
for all $x,y\in X$ and $t\in \left[ 0,T\right] $.
In addition, assume that $\left\{ S\left( t\right) ;t\geq 0\right\} $ is a
compact semigroup and condition \emph{(ii)} in Theorem \emph{\ref{teorema cu
Perov}} holds.

Then, for each $\beta \in X,$ problem \emph{(\ref{el})-(\ref{3.1'})} has at
least one solution $\left( x,y\right) $ in $C\left( \left[ 0,T\right]
;X\right) ^{2}$ with $y\left( 0\right) =\beta .$
\end{theorem}

\begin{proof}
We shall apply the Schauder's fixed point theorem to the operator $%
(N_{1},N_{2})$ on a closed convex bounded subset of $X_{1}\times X_{2}$ of
the form 
\begin{equation}
D_{\beta ;R_{1},R_{2}}:=\left\{ (x,y)\in C([0,T];X)^{2}\,:\,y\left( 0\right)
=\beta ,\ |x|_{0}\leq R_{1},\,|y|_{\theta }\leq R_{2}\right\} ,
\label{definitia D_alpha}
\end{equation}%
where the positive numbers $R_{1},$ $R_{2}$ will be determined in a such way
that the operator $(N_{1},N_{2})$ is invariant over $D_{\beta ;R_{1},R_{2}}$%
, i.e., 
\begin{equation*}
\left\vert N_{1}(x,y)\right\vert _{0}\leq R_{1},\ \left\vert
N_{2}(x,y)\right\vert _{\theta }\leq R_{2}\text{ \ whenever\ }|x|_{0}\leq
R_{1},\ \,|y|_{\theta }\leq R_{2}.
\end{equation*}%
Let $(x,y)\in D_{\beta ;R_{1},R_{2}}.$ Then, for all $t\in \lbrack 0,T]$, we
have%
\begin{align*}
\left\vert N_{1}(x,y)\left( t\right) \right\vert _{X}& \leq \frac{1}{a}%
C_{A}|x|_{0}+\frac{1}{a}C_{A}\left( 1+b\right) k\left\vert \beta \right\vert
_{X} \\
& +C_{A}\,\frac{k}{a}\int_{0}^{T}\left(
a_{21}|x(s)|_{X}+a_{22}|y(s)|_{X}+a_{23}\right) ds \\
& +C_{A}\left( 1+\frac{1}{a}\right) \int_{0}^{T}\left(
a_{11}|x(s)|_{X}+a_{12}|y(s)|_{X}+a_{13}\right) ds \\
& \leq m_{11}|x|_{0}+m_{12}\int_{0}^{T}e^{\theta s}e^{-\theta
s}|y(s)|_{X}+C_{1} \\
& \leq m_{11}R_{1}+m_{12}R_{2}\frac{e^{\theta T}-1}{\theta }+C_{1},
\end{align*}%
where 
\begin{equation*}
m_{11}=C_{A}\left( \frac{1}{a}+\left( 1+\frac{1}{a}\right) Ta_{11}+\frac{k}{a%
}Ta_{21}\right) ,\ \ m_{12}=C_{A}\left(\left( 1+\frac{1}{a}\right) a_{12}+\frac{k%
}{a}a_{22}\right),\text{ }
\end{equation*}%
and 
\begin{equation*}
C_{1}=\frac{1}{a}C_{A}\left( 1+b\right) k\left\vert \beta \right\vert
_{X}+C_{A}\frac{k}{a}Ta_{23}+C_{A}\left( 1+\frac{1}{a}\right) a_{13}.
\end{equation*}%
Hence, 
\begin{equation}
\left\vert N_{1}(x,y)\right\vert _{0}\leq m_{11}R_{1}+m_{12}R_{2}\frac{%
e^{\theta T}-1}{\theta }+C_{1}.  \label{cretere_N_1}
\end{equation}%
For the second operator $N_{2}$, we compute 
\begin{align*}
\left\vert N_{2}(x,y)\left( t\right) \right\vert _{X}& \leq C_{A}\left\vert
\beta \right\vert _{X}+C_{A}\int_{0}^{t}\left(
a_{21}|x(s)|_{X}+a_{22}|y(s)|_{X}+a_{23}\right) ds \\
& \leq \,C_{A}\,Ta_{21}|x|_{0}+a_{22}C_{A}\int_{0}^{t}e^{\theta s}e^{-\theta
s}|y(s)|_{X}ds+C_{2} \\
& \leq m_{21}|x|_{0}+m_{22}|y|_{\theta }\frac{e^{\theta t}-1}{\theta }+C_{2},
\end{align*}%
where $m_{21}=\,C_{A}\,Ta_{21}$, $m_{22}=C_{A}a_{22}$ and $%
C_{2}=C_{A}\left\vert \beta \right\vert _{X}+C_{A}Ta_{23}.$ Dividing by $%
e^{\theta t}$ and taking the supremum over $t\in \lbrack 0,T]$, we obtain 
\begin{equation}
\left\vert N_{2}(x,y)\right\vert _{\theta }\leq m_{21}R_{1}+m_{22}\frac{%
1-e^{-\theta T}}{\theta }R_{2}+C_{2}.  \label{r5}
\end{equation}%
In the matrix form, relations \eqref{cretere_N_1} and (\ref{r5}) read as 
\begin{equation*}
\begin{bmatrix}
\left\vert N_{1}(x,y)\right\vert _{0} \\ 
\left\vert N_{2}(x,y)\right\vert _{\theta }%
\end{bmatrix}%
\leq M(\theta )%
\begin{bmatrix}
R_{1} \\ 
R_{2}%
\end{bmatrix}%
+%
\begin{bmatrix}
C_{1} \\ 
C_{2}%
\end{bmatrix}%
,
\end{equation*}%
where $M(\theta )$ is given in (\ref{r0}). For the invariance condition it
suffices to have%
\begin{equation*}
M(\theta )%
\begin{bmatrix}
R_{1} \\ 
R_{2}%
\end{bmatrix}%
+%
\begin{bmatrix}
C_{1} \\ 
C_{2}%
\end{bmatrix}%
\leq 
\begin{bmatrix}
R_{1} \\ 
R_{2}%
\end{bmatrix}%
,
\end{equation*}%
or equivalently%
\begin{equation*}
\begin{bmatrix}
C_{1} \\ 
C_{2}%
\end{bmatrix}%
\leq \left( I-M\left( \theta \right) \right) 
\begin{bmatrix}
R_{1} \\ 
R_{2}%
\end{bmatrix}%
.
\end{equation*}%
From assumption (ii), since the matrix \( M(\theta) \) is convergent to zero, \( (I - M(\theta))^{-1} \) has nonnegative entries. Therefore, the last inequality is equivalent to 
\begin{equation}
(I-M(\theta ))^{-1}%
\begin{bmatrix}
C_{1} \\ 
C_{2}%
\end{bmatrix}%
\leq 
\begin{bmatrix}
R_{1} \\ 
R_{2}%
\end{bmatrix}%
,  \label{r11}
\end{equation}%
which clearly allows us to choose positive numbers $R_{1},R_{2}$ such that
the set $D_{\beta ;R_{1},R_{2}}$ is invariant under the operator $%
(N_{1},N_{2}).$

Note that, as the semigroup generated by $A$ is assumed to be compact, it follows from
standard arguments that $N_{1}$ and $N_{2}$ are completely continuous operators
(see, e.g., \cite{pcarte}). Consequently, Schauder's fixed point theorem
applies and guarantees the existence of at least one fixed point of $%
(N_{1},N_{2})$ in $D_{\beta ;R_{1},R_{2}}$.
\end{proof}

\begin{remark}
From the above result we have not only the existence of a solution, but also
its localization, namely%
\begin{equation}
|x(t)|_{X}\leq R_{1}\ \ \ \text{ and\ \ \ }|y(t)|_{X}\leq e^{\theta t}R_{2}\
\ \text{for all }t\in \lbrack 0,T],  \label{r10}
\end{equation}%
where $\theta $ is given in assumption (ii), and $%
R_{1},R_{2}$ satisfy (\ref{r11}). Note that, due to the use of the Bielecki
norm on the second component, we only have an exponential localization for $%
y $.
\end{remark}

\subsection{Existence via Avramescu's fixed point theorem}

\phantom{}

We apply Avramescu's fixed point theorem to obtain a fixed point for the
operator $(N_{1},N_{2})$. It is noteworthy that, given the fixed point
obtained in Theorem \ref{t33} below, the uniqueness of the second component
is guaranteed by Banach's principle, while both components $x$ and $y$
satisfy (\ref{r10}) with $\theta $ given in assumption (ii) of Theorem \ref%
{t31}, and $R_{1},R_{2}$ satisfy (\ref{r11}).

\begin{theorem}
\label{t33} Assume that there are nonnegative constants $%
a_{11},a_{12},a_{13},a_{21},a_{22}$ such that 
\begin{align}
& |F(x,y)|_{X}\leq a_{11}|x|_{X}+a_{12}|y|_{X}+a_{13},  \label{ac1} \\
& |G(x,y)-G(\overline{x},\overline{y})|_{X}\leq a_{21}|x-\overline{x}%
|_{X}+a_{22}|y-\overline{y}|_{X},  \notag
\end{align}%
for all $x,\overline{x},y,\overline{y}\in X$ and $t\in \left[ 0,T\right] .$
In addition assume that $\left\{ S\left( t\right) ;t\geq 0\right\} $ is a
compact semigroup and condition \emph{(ii)} in Theorem \emph{\ref{t31}}
holds.

Then, for each $\beta \in X,$ problem \emph{(\ref{el})-(\ref{3.1'})} has at
least one solution $\left( x,y\right) $ in $C\left( \left[ 0,T\right]
;X\right) ^{2}$ with $y\left( 0\right) =\beta .$
\end{theorem}

\begin{proof}
Let us consider the sets, 
\begin{align*}
& D_{1}:=\left\{ x\in C([0,T];X)\,:\,|x|_{0}\leq R_{1}\right\} , \\
& D_{2}:=\left\{ y\in C([0,T];X)\,:\,y\left( 0\right) =\beta ,\ |y|_{\theta
}\leq R_{1}\right\} ,
\end{align*}%
where $R_{1},R_{2}$ are positive numbers which are chosen below. We observe
that, from (\ref{ac1}), function $G$ satisfies the growth condition%
\begin{equation*}
|G(x,y)|_{X}\leq a_{21}|x|_{X}+a_{22}|y|_{X}+a_{23},
\end{equation*}%
where $a_{23}=\left\vert G\left( 0,0\right) \right\vert _{X}.$ Consequently,
performing the same estimates as in the proof of Theorem \ref{teorema cu
Schauder} and using (ii), we conclude that%
\begin{equation*}
\left\vert N_{1}(x,y)\right\vert _{0}\leq R_{1},\ \ \ \left\vert
N_{2}(x,y)\right\vert _{\theta }\leq R_{2},
\end{equation*}%
for all $(x,y)\in D_{1}\times D_{2}$, where $R_{1},R_{2}$ are obtained as in
the proof of the previous theorem. This guarantees that $N_{i}\left(
D_{1}\times D_{2}\right) \subset D_{i}$ for $i=1,2.$

Since the matrix $M(\theta )$ converges to zero, the diagonal elements are
strictly less than $1$, hence
\begin{equation*}
\ a_{22}C_{A}\frac{1-e^{-\theta T}}{\theta }<1.
\end{equation*}
This guarantees that $N_{2}(x,\cdot )$ is a contraction for all $x\in
C([0,T];X)$. Indeed, following a similar reasoning as in the proof of
Theorem \ref{t31}, we deduce 
\begin{equation*}
\left\vert N_{2}(x,y)-N_{2}(x,\overline{y})\right\vert _{\theta }\leq
a_{22}C_{A}\frac{1-e^{-\theta T}}{\theta }|y-\overline{y}|_{\theta },
\end{equation*}%
for all $x\in D_{1}$ and $y,\overline{y}\in D_{2}$. In addition, as above, $%
N_{1}\left( D_{1}\times D_{2}\right) $ is relatively compact in $C\left( %
\left[ 0,T\right] ;X\right) .$ Therefore, Avramescu's theorem applies and
guarantees the existence of a pair $(x,y)\in D_{1}\times D_{2}$ such that $%
N_{1}(x,y)=x$ and $N_{2}(x,y)=y$. 
\end{proof}

\section{Application}

The results established in Section 4 can be easily applied to the diffusion
system, 
\begin{equation}
\begin{cases}
u_{t}=\Delta u+f(u,v) \\ 
v_{t}=\Delta v+g(u,v)\,\text{ in }\Omega , \\ 
u=v=0\,\text{ on }\partial \Omega%
\end{cases}%
(t\in \lbrack 0,T]),  \label{aplicatie}
\end{equation}%
with the controllability condition $\ $%
\begin{equation}
u(T)-au\left( 0\right) =k\left( v(T)-bv\left( 0\right) \right) ,
\label{5.1'}
\end{equation}%
where $a,b,k>0.$

The functions $f,g\colon \mathbb{R}^{2}\rightarrow \mathbb{R}$ are assumed
to be continuous, and there exist nonnegative constants $%
a_{11},a_{12},a_{21},a_{22},c_{f},c_{g}$ such that one of the following
three conditions is satisfied: 
\begin{align*}
& (\text{c}_{1}):%
\begin{cases}
\left\vert f(p,q)-f(\overline{p},\overline{q})\right\vert \leq a_{11}|p-%
\overline{p}|+a_{12}|q-\overline{q}| \\ 
\left\vert g(p,q)-g(\overline{p},\overline{q})\right\vert \leq a_{21}|p-%
\overline{p}|+a_{22}|q-\overline{q}|%
\end{cases}
\\
& (\text{c}_{2}):%
\begin{cases}
\left\vert f(p,q)\right\vert \leq a_{11}|p|+a_{12}|q|+c_{f} \\ 
\left\vert g(p,q)\right\vert \leq a_{21}|p|+a_{22}|q|+c_{g}%
\end{cases}
\\
& (\text{c}_{3}):%
\begin{cases}
\left\vert f(p,q)\right\vert \leq a_{11}|p|+a_{12}|q|+c_{f} \\ 
\left\vert g(p,q)-g(\overline{p},\overline{q})\right\vert \leq a_{21}|p-%
\overline{p}|+a_{22}|q-\overline{q}|%
\end{cases}%
\end{align*}%
for all $p,q,\overline{p},\overline{q}\in \mathbb{R}$.

Here, $\Omega \subset \mathbb{R}^{n}$ is a bounded open set, $X$ is the
Banach space $L^{2}(\Omega )$ endowed with the usual $|\cdot |_{L^{2}}$
norm, $A=\Delta $, with 
\begin{equation*}
D(A)=\left\{ u\in H_{0}^{1}(\Omega )\,:\,\Delta u\in L^{2}(\Omega )\right\} .
\end{equation*}%
Note that $A$ is the infinitesimal generator of a compact semigroup of
contractions in $L^{2}(\Omega )$ (cf, \cite[Theorem~7.2.5]{v}). Hence, the
constant $C_{A}$ given in \eqref{marginire semigrup} is $1.$ Also $F$ and $G$
are the superposition operators $\ F,G:L^{2}\left( \Omega \right)
^{2}\rightarrow L^{2}\left( \Omega \right) ,$ 
\begin{equation*}
F\left( u,v\right) \left( x\right) =f\left( u\left( x\right) ,v\left(
x\right) \right) ,\ \ \ G\left( u,v\right) \left( x\right) =g\left( u\left(
x\right) ,v\left( x\right) \right) \ \ \ \left( x\in \Omega \right) .
\end{equation*}

Simple computations show that if either $($c$_{1})$, $($c$_{2})$, or $($c$%
_{3})$ holds, then assumption (i) of Theorem \ref{t31}, Theorem \ref{teorema
cu Schauder}, or Theorem \ref{t33} is satisfied with $a_{13}=c_{f}m(\Omega
)^{1/2},\ a_{23}=c_{g}m(\Omega )^{1/2}$ and the constants $%
a_{11},a_{12},a_{21},a_{22}$.

Additionally, if we assume there exists $\theta \geq 0$ such that the matrix 
\begin{equation*}
M(\theta ):=%
\begin{bmatrix}
\frac{1}{a}+\left( 1+\frac{1}{a}\right) Ta_{11}+\frac{k}{a}Ta_{21} & \left(
\left( 1+\frac{1}{a}\right) a_{12}+\frac{k}{a}a_{22}\right) \frac{e^{\theta
T}-1}{\theta } \\[3pt] 
Ta_{21} & a_{22}\,\frac{1-e^{-\theta T}}{\theta }%
\end{bmatrix}%
\end{equation*}%
converges to zero, then condition (ii) of the Theorem \ref{t31} is verified.
Consequently, depending on the conditions imposed ((c$_{1}$), (c$_{2}$), or
(c$_{3}$)), there exists a weakly solution for the mutual control problem %
\eqref{aplicatie}-(\ref{5.1'}), which is either unique; localized; or unique
in one component and localized in both.


\begin{thebibliography}{99}
\bibitem{a} C. Avramescu, \textit{Asupra unei teoreme de punct fix }(in
Romanian), St. Cerc. Mat. \textbf{22 }(2) (1970), 215--221.

\bibitem{bg} M. Beldinski and M. Galewski, \textit{Nash type equilibria for
systems of non-potential equations}, Appl. Math. Comput. \textbf{385}
(2020), pp. 125456.

\bibitem{c} M. Coron, \textit{Control and Nonlinearity}, AMS, Providence,
2007.

\bibitem{d} K. Deimling, \textit{Nonlinear Functional Analysis}, Springer,
Berlin, 1985.

\bibitem{g} A. Granas, \textit{Fixed Point Theory}, Springer, New York, 2003.

\bibitem{kr} M.A. Krasnoselskii, \textit{Some problems of nonlinear analysis}%
, Amer. Math. Soc. Transl. \textbf{10} (1958), 345--409.

\bibitem{mr} P. Magal and S. Ruan, \textit{Theory and Applications of
Abstract Semilinear Cauchy Problems}, Springer, 2018.

\bibitem{park} S. Park, \textit{Generalizations of the Nash equilibrium
theorem in the KKM theory}, Fixed Point Theor. Appl. \textbf{2010} (2010).

\bibitem{pmcz} R. Precup, \textit{The role of matrices that are convergent
to zero in the study of semilinear operator systems}, Math. Comput. Model. 
\textbf{49} (2009), 703--708.

\bibitem{p1} R. Precup, \textit{Nash-type equilibria and periodic solutions
to nonvariational systems}, Adv. Nonlinear Anal. \textbf{3} (2014), no. 4,
197--207.

\bibitem{p2} R. Precup, \textit{A critical point theorem in bounded convex
sets and localization of Nash-type equilibria of nonvariational systems}, J.
Math. Anal. Appl. \textbf{463} (2018), 412--431.

\bibitem{pn} R. Precup, \textit{On some applications of the controllability
principle for fixed point equations}, Results Appl. Math. \textbf{13 }
(2022) 100236, 1--7.

\bibitem{pcarte} R. Precup, \textit{Methods in Nonlinear Integral Equations}%
, Dordrecht, Springer, 2002.

\bibitem{ps} R. Precup and A. Stan, \textit{Stationary Kirchhoff equations
and systems with reaction terms}, AIMS Mathematics, \textbf{7} (2022), Issue
8, 15258--15281.

\bibitem{ps2} R. Precup and A. Stan, \textit{Linking methods for
componentwise variational systems}, Results Math. \textbf{78} (2023), 1-25.

\bibitem{ps3} R. Precup and A. Stan, \textit{A mutual control problem for
semilinear systems via fixed point approach, }submitted.

\bibitem{s} A. Stan, \textit{Nonlinear systems with a partial Nash type
equilibrium}, Studia Univ. Babe\c{s}-Bolyai Math. \textbf{66} (2021),
397--408.

\bibitem{s2} A. Stan, \textit{Nash equilibria for componentwise variational
systems}, J. Nonlinear Funct. Anal. \textbf{6} (2023), 1-10.

\bibitem{s3} A. Stan, \textit{Localization of Nash-type equilibria for
systems with partial variational structure}, J. Numer. Anal. Approx. Theory, 
\textbf{52} (2023) , 253--272.

\bibitem{v} I.I. Vrabie, $C_{0}$\textit{-Semigroups and Applications },
Elsevier, Amsterdam, 2003.

\bibitem{z} J. Zabczyk, \textit{Mathematical Control Theory}, Springer,
Cham, 2020.
\end{thebibliography}
\end{document}